\documentclass[11pt,letterpaper]{amsart}

\usepackage{amsmath}
\usepackage{amssymb}
\usepackage{amsthm}
\usepackage{geometry}
\usepackage{xcolor}
\definecolor{babyblue}{RGB}{35,135,220}
\usepackage{tikz}
\usepackage{comment}
\usepackage{float}

\usepackage[
    backend=biber,
    style=numeric,
    sorting=none,
    maxbibnames=99,
    doi=true,
    eprint=false,
    isbn=true,
    url=true
]{biblatex}
\usepackage[
    colorlinks=true,
    citecolor=babyblue,
    linkcolor=babyblue,
    urlcolor=babyblue,
]{hyperref}

\usepackage{pgfplots}
\pgfplotsset{compat=1.18}

\usetikzlibrary{arrows.meta}
\theoremstyle{plain}
\newtheorem{theorem}{Theorem}[section]
\newtheorem{lemma}[theorem]{Lemma}

\newtheorem{corollary}[theorem]{Corollary}
\newtheorem{conjecture}[theorem]{Conjecture}

\newtheorem{definition}[theorem]{Definition}

\newtheorem{remark}[theorem]{Remark}
\newtheorem{observation}[theorem]{Observation}

\newcommand{\Z}{\mathbb Z}

\newcommand{\defect}{\operatorname{def}}

\title[A Hall Condition for Hypergraphs]
{A Necessary and Sufficient Hall Condition for
Hypergraphs}

\author{Xiaoyao Huang}
\address{Department of Mathematics, University of Michigan, Ann Arbor, MI 48109}
\email{xyrushac@umich.edu}

\date{\today}

\begin{document}

\begin{abstract}
We prove a necessary and sufficient Hall condition for a family \(\mathcal A=(A_e)_{e\in E(G)}\) of hypergraphs indexed by the edges of a forest \(G\). 
This restriction on the index graph is sharp.

The loop-only case recovers the classical Hall's Theorem with multiplicities for arbitrary finite set systems, while the loopless case shows that full rainbow matching is polynomial time solvable under this forest structure, although the problem is NP-complete in general.

As another application, we prove every \(5\)-tough chordal graph is Hamilton-connected, improving toughness bounds of \(18\) for Hamiltonicity (1998) and \(10\) for Hamilton-connectedness (2017).
\end{abstract}

\maketitle


\section{Introduction}
Hall's Marriage Theorem characterizes when a family of sets has a \emph{system of distinct representatives} (SDR).
An SDR for a family of hypergraphs analogously chooses one hyperedge from each member so that the chosen hyperedges are pairwise vertex disjoint.
For general families of hypergraphs, the analogous Hall condition remains necessary but is not sufficient.
Aharoni and Haxell~\cite{AharoniHaxell} proved a sufficient Hall condition based on matchings in unions of subfamilies.
We identify a natural class in which the hypergraphs share a common ground set and are indexed by the edges of a tree, while each ground set element is associated with a subtree of that tree.
In this class, our Hall condition is both necessary and sufficient.

More precisely, let \(T\) be a finite tree and let \(U\) be a finite set of \emph{resources}.
Associate with every \(x\in U\) a subtree \(Q_x\) of \(T\), called the \emph{support} of \(x\).
For \(t\in V(T)\), define the \emph{bag} at \(t\) by
\[
  P_t:=\{x\in U:t\in V(Q_x)\}.
\]
A family \(\mathcal A=(A_e)_{e\in E(T)}\) is a \emph{tree-structured hypergraph system} if, for every \(e=tu\), \(A_e\) is a hypergraph on \(U\) such that its loops are the resources in \(P_t\cap P_u\), and every nonloop hyperedge is of the form \(\{x,y\}\), where \(x\in P_t\) and \(y\in P_u\).
Given \(q:E(T)\to\mathbb Z_{>0}\), a \emph{\(q\)-SDR} chooses \(q_e:=q(e)\) representatives from every \(A_e\), with all chosen representatives pairwise vertex disjoint.
The precise definitions are given in Definitions~\ref{def:tree-system} and~\ref{def:qSDR}.

For every \(e\in E(T)\), introduce \(q_e\) vertices \(z_{e,1},\ldots,z_{e,q_e}\).
Let \(H(\mathcal A,q)\) be the hypergraph on
\[
  U\cup\{z_{e,i}:e\in E(T),\ 1\le i\le q_e\}.
\]
For every \(x\in U\), include the singleton hyperedge \(\{x\}\).
For every \(e\in E(T)\), \(1\le i\le q_e\), and \(a\in E(A_e)\), also include the hyperedge \(\{z_{e,i}\}\cup a\), and include no others.

Our main theorem is the following Hall-type characterization.
\begin{theorem}
\label{thm:qSDR}
Let \(\mathcal A=(A_e)_{e\in E(T)}\) be a tree-structured hypergraph system, and let \(q:E(T)\to\mathbb Z_{>0}\) be a multiplicity function.
The following statements are equivalent:
\begin{itemize}
  \item \(\mathcal A\) admits a \(q\)-SDR.
    \hfill\textup{[SDR]}
  \item The hypergraph \(H(\mathcal A,q)\) has a perfect matching.
    \hfill\textup{[Matching]}
  \item \(\displaystyle \nu\left(\bigcup_{e\in F}E(A_e)\right) \ge \sum_{e\in F}q_e \) for every \(F\subseteq E(T)\).
    \hfill\textup{[Hall]}
\end{itemize}
\end{theorem}

\begin{observation}
\label{obs:hall-special-case}
Theorem~\ref{thm:qSDR} contains the classical version of Hall's Marriage Theorem with multiplicities as a special case
where all representatives are loops.
\end{observation}

Indeed, every finite set system is realized by a loop-only system on a star.

For example, let \(E=\{e_1,e_2,e_3,e_4\}\) and \(U=\{w,x,y,z\}\), with all multiplicities equal to one.
\begin{figure}[ht]
\centering
\begin{tikzpicture}[
  treevertex/.style={circle,fill=black,minimum size=1.8mm,inner sep=0pt},
  graphvertex/.style={circle,draw=black,fill=white,minimum size=5.5mm,inner sep=0pt},
  every node/.style={font=\scriptsize}
]
  \node[anchor=east,font=\bfseries] at (-5.8,1.75) {(a)};
  \node[anchor=west,font=\itshape] at (-5.6,1.75) {The star \(T\) and its bags};
  \node[treevertex] (center) at (-2.7,0) {};
  \node[treevertex] (leaf1) at (-4.3,1.2) {};
  \node[treevertex] (leaf2) at (-4.3,.4) {};
  \node[treevertex] (leaf3) at (-4.3,-.4) {};
  \node[treevertex] (leaf4) at (-4.3,-1.2) {};
  \draw[black,line width=.55pt] (center)--node[above] {\(e_1\)} (leaf1);
  \draw[black,line width=.55pt] (center)--node[above] {\(e_2\)} (leaf2);
  \draw[black,line width=.55pt] (center)--node[below] {\(e_3\)} (leaf3);
  \draw[black,line width=.55pt] (center)--node[below] {\(e_4\)} (leaf4);
  \node[anchor=west] at (-2.55,0) {\(U=\{w,x,y,z\}\)};
  \node[anchor=east] at (-4.45,1.2) {\(S_{e_1}=\{w,x,y\}\)};
  \node[anchor=east] at (-4.45,.4) {\(S_{e_2}=\{x\}\)};
  \node[anchor=east] at (-4.45,-.4) {\(S_{e_3}=\{x,y,z\}\)};
  \node[anchor=east] at (-4.45,-1.2) {\(S_{e_4}=\{z,w\}\)};

  \node[anchor=east,font=\bfseries] at (1.45,1.75) {(b)};
  \node[anchor=west,font=\itshape] at (1.65,1.75) {The bipartite graph};
  \node at (1.5,1.3) {Indices};
  \node at (4.5,1.3) {Resources};
  \coordinate (be1) at (1.5,.8);
  \coordinate (be2) at (1.5,.2);
  \coordinate (be3) at (1.5,-.4);
  \coordinate (be4) at (1.5,-1);
  \coordinate (bw) at (4.5,.8);
  \coordinate (bx) at (4.5,.2);
  \coordinate (by) at (4.5,-.4);
  \coordinate (bz) at (4.5,-1);
  \draw[black,line width=.5pt] (be1)--(bx);
  \draw[black,line width=.5pt] (be1)--(by);
  \draw[black,line width=.5pt] (be3)--(bx);
  \draw[black,line width=.5pt] (be3)--(bz);
  \draw[black,line width=.5pt] (be4)--(bw);
  \draw[babyblue,line width=1pt] (be1)--(bw);
  \draw[babyblue,line width=1pt] (be2)--(bx);
  \draw[babyblue,line width=1pt] (be3)--(by);
  \draw[babyblue,line width=1pt] (be4)--(bz);
  \node[graphvertex] at (be1) {\(e_1\)};
  \node[graphvertex] at (be2) {\(e_2\)};
  \node[graphvertex] at (be3) {\(e_3\)};
  \node[graphvertex] at (be4) {\(e_4\)};
  \node[graphvertex] at (bw) {\(w\)};
  \node[graphvertex] at (bx) {\(x\)};
  \node[graphvertex] at (by) {\(y\)};
  \node[graphvertex] at (bz) {\(z\)};
\end{tikzpicture}
\caption{The star \(T\) and the corresponding bipartite graph; the blue edges form a perfect matching corresponding to the SDR \(\{w\in S_{e_1},x\in S_{e_2},y\in S_{e_3},z\in S_{e_4}\}\).}
\label{fig:hall-special-case}
\end{figure}
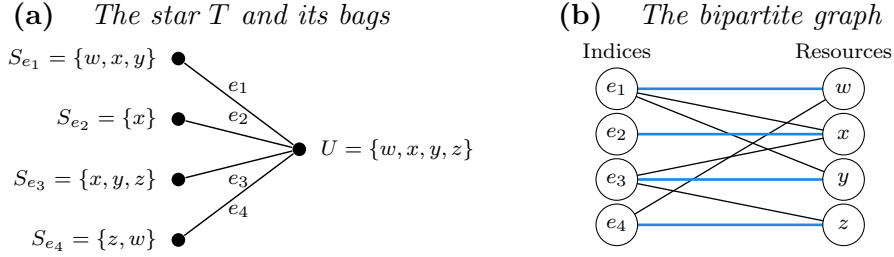


The tree \(T\) is a \emph{join tree} for the bag hypergraph \((P_t)_{t\in V(T)}\)~\cite{BeeriFaginMaierYannakakis}: the bags containing each \(x\in U\) are indexed by the vertices of the subtree \(Q_x\).
Conceptually, this representation is related to the point-tree hypergraphs of Aharoni, Berger and Ziv~\cite{AharoniBergerZiv}, in which each hyperedge consists of one point together with the vertices of a subtree.
Their theorem is a tree version of K\H{o}nig's theorem.
In this sense, Theorem~\ref{thm:qSDR} is a tree version of Hall's theorem for hypergraph SDRs.
Here, each \(e\in E(T)\) indexes a hypergraph \(A_e\) of singleton and pair representatives, and \(q_e\) pairwise vertex disjoint representatives must be chosen.

For each edge of \(T\), the proof recursively constructs a \emph{reservation matroid} recording which sets of resources can be reserved on one side of the edge; its dual, the \emph{consumption matroid}, records the minimal sets of resources consumed there.
Methodologically, this construction is related to Schrijver's linking systems~\cite{SchrijverLinking}, which transfer matroids through bipartite and more general linking relations.
The key additional step is that any obstruction to the recursive construction of these matroids yields a set \(F\subseteq E(T)\) with
\(\nu\bigl(\bigcup_{e\in F}E(A_e)\bigr)<\sum_{e\in F}q_e\), violating the Hall condition.

More generally, a system on a finite simple graph \(G\) is defined as in Definition~\ref{def:tree-system}, with \(G\) in place of tree \(T\) and each support \(Q_x\) required to be a connected subgraph of \(G\).
Together with Theorem~\ref{thm:qSDR}, the following theorem gives a complete classification of the index graphs for which [Hall] universally characterizes the existence of \(q\)-SDRs.
\begin{theorem}
\label{thm:forest-restriction}
Let \(G\) be a finite simple graph.
Then \(G\) is a forest if and only if, for every system on \(G\) and every multiplicity function \(q:E(G)\to\mathbb Z_{>0}\), the [Hall] condition in Theorem~\ref{thm:qSDR} implies the existence of a \(q\)-SDR.
\end{theorem}

Theorem~\ref{thm:qSDR} also yields a min--max formula and a polynomial time feasibility test; see Corollaries~\ref{cor:deficiency} and~\ref{cor:polynomial-algorithm}.

Our first application concerns chordal graphs.
Chv\'atal introduced \emph{toughness} to measure how hard it is to break a graph into many components~\cite{Chvatal}.
Chen, Jacobson, K\'{e}zdy and Lehel~\cite{ChenJacobsonKezdyLehel} (1998) obtained the toughness bound \(18\) for Hamiltonicity of chordal graphs, and Kabela and Kaiser~\cite{KabelaKaiser} (2017) obtained the bound \(10\) for Hamilton-connectedness.
By replacing the sufficient Hall condition \cite{AharoniHaxell} used in the latter work with the necessary and sufficient condition in Theorem~\ref{thm:qSDR}, we lower this bound to \(5\).
\begin{theorem}
\label{thm:five-tough-hamilton-connected}
Every \(5\)-tough chordal graph on at least three vertices is Hamilton-connected.
\end{theorem}

Chordal graphs are one specific class to which Theorem~\ref{thm:qSDR} can be applied.
More broadly, the if and only if Hall condition serves as a general tool that can be specialized, under additional structural assumptions, to problems in other graph classes.

Our second application concerns full rainbow matchings.
The loopless case of Theorem~\ref{thm:qSDR} yields a polynomial time algorithm for full rainbow matching (Corollary~\ref{cor:prescribed-colors}), while its loop-only case recovers Hall's Marriage Theorem with multiplicities (Observation~\ref{obs:hall-special-case}).
\section{\texorpdfstring{A Hall-type Theorem for \(q\)-SDRs}{A Hall-type Theorem for q-SDRs}}
This section proves Theorem~\ref{thm:qSDR}.
We retain the tree \(T\), resource set \(U\), supports \((Q_x)_{x\in U}\), and bags \((P_t)_{t\in V(T)}\) defined in the introduction.

\subsection{Preliminaries}
Throughout this section, a loop is regarded as a singleton hyperedge, both for matchings and for vertex covers.
For a hypergraph \(H\) and \(X\subseteq V(H)\), write \(H-X\) for the hypergraph obtained by deleting the vertices in \(X\) and every hyperedge meeting \(X\).
We write \(\nu(H)\) for the maximum size of a matching in a hypergraph \(H\) and \(\tau(H)\) for the minimum size of a vertex cover of \(H\).
For a matroid \(M\), let \(r_M\) denote its rank function.
For a matroid \(M\) on ground set \(E\) and \(0\le k\le r_M(E)\), let \(\operatorname{Tr}_k(M)\) denote the truncation of \(M\) to rank \(k\), whose independent sets are the independent sets of \(M\) of size at most \(k\).

We also use the following three standard facts from matroid theory:

For matroids \(J_1,\ldots,J_m\) on a common ground set, we use the finite matroid union rank formula~\cite{EdmondsFulkerson}
\begin{equation}
\label{eq:matroid-union-rank}
  r_{J_1\vee\cdots\vee J_m}(X)=\min_{W\subseteq X}\left(|X\setminus W|+\sum_{i=1}^{m}r_{J_i}(W)\right).
\end{equation}

The Rado--Perfect theorem~\cite{Rado,Perfect} states that, if \(K\) is bipartite with classes \(L,R\) and \(M\) is a matroid on \(R\), then the subsets of \(L\) matchable to an \(M\)-independent subset of \(R\) form a matroid \(M[K]\), with
\begin{equation}
\label{eq:rado-perfect-rank}
  r_{M[K]}(H)=\min_{A\subseteq H}\bigl(|H\setminus A|+r_M(N_K(A))\bigr) \qquad \text{for all } H\subseteq L.
\end{equation}

For a matroid \(J\) on \(E\), we also use the dual rank identity
\begin{equation}
\label{eq:dual-rank}
  r_{J^*}(X)=|X|-r_J(E)+r_J(E\setminus X) \qquad \text{for all } X\subseteq E.
\end{equation}

The following definition describes possible choices of representatives on a tree: each resource is available throughout a connected subtree, and the hypergraph associated with an edge \(e=tu\) records choices involving resources available at \(t\) and \(u\).

\begin{definition}
\label{def:tree-system}
A family \(\mathcal A=(A_e)_{e\in E(T)}\) is a \emph{tree-structured hypergraph system} if, for every \(e=tu\), \(A_e\) is a hypergraph on \(U\), and the following conditions hold:
\begin{enumerate}
  \item[(i)] the resources carrying loops in \(A_e\) are the elements of \(P_t\cap P_u\);
  \item[(ii)] every nonloop hyperedge of \(A_e\) is of the form \(\{x,y\}\), where \(x\in P_t\) and \(y\in P_u\).
\end{enumerate}
\end{definition}

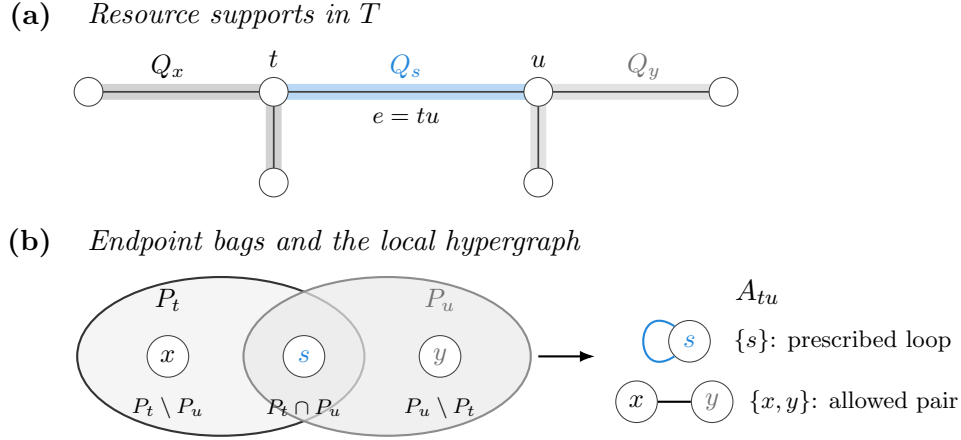
\begin{figure}[ht]
\centering
\begin{tikzpicture}[
  treevertex/.style={
    circle,draw=black!75,fill=white,minimum size=3.8mm,inner sep=0pt
  },
  resource/.style={
    circle,draw=black!75,fill=white,minimum size=5.5mm,inner sep=0pt
  },
  every node/.style={font=\small}
]
  \node[anchor=east,font=\bfseries] at (-4.55,2.75) {(a)};
  \node[anchor=west,font=\itshape] at (-4.35,2.75)
    {Resource supports in \(T\)};

  \draw[black,line width=6pt,opacity=.18,line cap=round]
    (-4.2,1.75)--(-1.75,1.75)--(-1.75,.55);
  \draw[babyblue,line width=6pt,opacity=.32,line cap=round]
    (-1.75,1.75)--(1.75,1.75);
  \draw[black!45,line width=6pt,opacity=.25,line cap=round]
    (1.75,.55)--(1.75,1.75)--(4.2,1.75);

  \draw[black!75,line width=.65pt]
    (-4.2,1.75)--(-1.75,1.75)--(1.75,1.75)--(4.2,1.75)
    (-1.75,1.75)--(-1.75,.55)
    (1.75,1.75)--(1.75,.55);
  \node[treevertex] at (-4.2,1.75) {};
  \node[treevertex,label=above:\(t\)] at (-1.75,1.75) {};
  \node[treevertex,label=above:\(u\)] at (1.75,1.75) {};
  \node[treevertex] at (4.2,1.75) {};
  \node[treevertex] at (-1.75,.55) {};
  \node[treevertex] at (1.75,.55) {};

  \node[text=black] at (-3.15,2.08) {\(Q_x\)};
  \node[text=babyblue] at (0,2.08) {\(Q_s\)};
  \node[text=black!55] at (3.15,2.08) {\(Q_y\)};
  \node[font=\footnotesize] at (0,1.43) {\(e=tu\)};

  \node[anchor=east,font=\bfseries] at (-4.55,-.25) {(b)};
  \node[anchor=west,font=\itshape] at (-4.35,-.25)
    {Endpoint bags and the local hypergraph};

  \draw[black!80,thick,fill=black!4]
    (-2.45,-1.75) ellipse (1.9cm and 1.05cm);
  \draw[black!45,thick,fill=black!8,fill opacity=.72]
    (-.25,-1.75) ellipse (1.9cm and 1.05cm);
  \node[text=black] at (-3.15,-1) {\(P_t\)};
  \node[text=black!55] at (.45,-1) {\(P_u\)};

  \node[resource,text=black] at (-3.15,-1.75) {\(x\)};
  \node[resource,text=babyblue] at (-1.35,-1.75) {\(s\)};
  \node[resource,text=black!55] at (.45,-1.75) {\(y\)};
  \node[font=\scriptsize] at (-3.15,-2.43) {\(P_t\setminus P_u\)};
  \node[font=\scriptsize] at (-1.35,-2.43) {\(P_t\cap P_u\)};
  \node[font=\scriptsize] at (.45,-2.43) {\(P_u\setminus P_t\)};

  \draw[-{Latex[length=2mm]},thick] (1.75,-1.75)--(2.5,-1.75);
  \node[font=\bfseries] at (4.65,-.9) {\(A_{tu}\)};

  \node[resource,text=babyblue] (loops) at (3.75,-1.55) {\(s\)};
  \draw[babyblue,thick]
    (loops.135) .. controls +(-5.5mm,3.6mm) and +(-5.5mm,-3.6mm) ..
    (loops.225);
  \node[anchor=west,font=\footnotesize] at (4.2,-1.55)
    {\(\{s\}\): prescribed loop};

  \node[resource,text=black] (pairx) at (3.05,-2.35) {\(x\)};
  \node[resource,text=black!55] (pairy) at (4.05,-2.35) {\(y\)};
  \draw[thick] (pairx)--(pairy);
  \node[anchor=west,font=\footnotesize] at (4.4,-2.35)
    {\(\{x,y\}\): allowed pair};
\end{tikzpicture}
\caption{Local structure of a tree-structured hypergraph system at \(e=tu\).
Note that \(\{s\}\) is a prescribed loop of \(A_{tu}\), whereas \(\{x,y\}\) is an allowed, but not required, nonloop hyperedge.}
\label{fig:tree-system}
\end{figure}

\begin{definition}
\label{def:qSDR}
Let \(\mathcal A=(A_e)_{e\in E(T)}\) be a tree-structured hypergraph system, and let \(q:E(T)\to\Z_{>0}\) be a multiplicity function.
A selected hyperedge \emph{consumes} its resources.
A \emph{\(q\)-SDR} for \(\mathcal A\) chooses \(q_e := q(e)\) hyperedges of \(A_e\) for every \(e\in E(T)\), with all chosen hyperedges pairwise vertex disjoint; that is, no resource is consumed twice.
For \(F\subseteq E(T)\), define the hypergraph \(A(F)\) on \(U\) and the total multiplicity \(q(F)\), respectively, by
\[
  E(A(F))=\bigcup_{e\in F}E(A_e), \qquad q(F)=\sum_{e\in F}q_e.
\]
\end{definition}

\begin{proof}[Proof of Observation~\ref{obs:hall-special-case}]
Let \((S_e)_{e\in E}\) be a finite set system with union \(U\) and multiplicities \((q_e)_{e\in E}\).
Let \(T\) be the star with center \(c\), leaves \(v_e\) indexed by \(e\in E\), and edge \(e=cv_e\).
For each \(x\in U\), let \(Q_x\) be the subtree consisting of \(c\), the leaves \(v_e\) for which \(x\in S_e\), and their incident edges.
Then \(P_c=U\) and \(P_{v_e}=S_e\).
For each \(e\in E\), let \(A_e\) consist of the loops \(\{x\}\) with \(x\in S_e\).
Since \(P_c\cap P_{v_e}=S_e\), \(\mathcal A\) is a tree-structured hypergraph system.
A \(q\)-SDR for this system is therefore a choice of \(q_e\) representatives from each \(S_e\), with all chosen representatives distinct.
Moreover,
\[
  \nu(A(F))=\left|\bigcup_{e\in F}S_e\right| \qquad \text{for every \(F\subseteq E\)}.
\]
Thus [Hall] in Theorem~\ref{thm:qSDR} is exactly the standard Hall condition with multiplicities.
\end{proof}

For a rooted tree \(T\), write \(s\preceq t\) if \(s\) is an ancestor of \(t\).
For \(t\in V(T)\), let \(T_t\) be the descendant subtree induced by \(\{u:t\preceq u\}\) and write \(E_t:=E(T_t)\).

\subsection{Recursive Setup and Proof of Theorem~\ref{thm:qSDR}}
\begin{definition}
\label{def:reservation}
We say a subset \(X\) of \(P_t\) is \emph{reservable} if the subsystem \((A_e)_{e\in E_t}\), with multiplicities \((q_e)_{e\in E_t}\), admits a \(q\)-SDR that consumes no resource of \(X\).
Define the \emph{reservation family} at \(t\) by
\[
  \mathcal I_t := \left\{X\subseteq P_t:\text{\(X\) is reservable}\right\}.
\]
Whenever \(\mathcal I_t\) is the independent set family of a matroid, we denote this matroid by \(M_t\) and call it the \emph{reservation matroid} at \(t\).
Its dual \(C_t:=M_t^*\) is called the \emph{consumption matroid} at \(t\).
\end{definition}

\begin{lemma}
\label{lem:semantic}
Let \(t\in V(T)\), and suppose that the reservation family \(\mathcal I_t\) is the family of independent sets of \(M_t\).
Then \(B\subseteq P_t\) is a basis of \(C_t=M_t^*\) if and only if \(B\) is inclusion-minimal among the sets of resources in \(P_t\) consumed by \(q\)-SDRs for \(E_t\).
\end{lemma}

\begin{proof}
A set \(B\subseteq P_t\) is a basis of \(C_t\) if and only if \(P_t\setminus B\) is a basis of \(M_t\), equivalently, an inclusion-maximal set of \(\mathcal I_t\).
By the definition of \(\mathcal I_t\), this holds if and only if some \(q\)-SDR for \(E_t\) consumes exactly \(B\) among the resources of \(P_t\), while no \(q\)-SDR consumes exactly a proper subset of \(B\) there.
This is precisely the inclusion-minimality of \(B\).
\end{proof}

Definition~\ref{def:reservation} and Lemma~\ref{lem:semantic} also apply to the subsystems in the proof of Theorem~\ref{thm:qSDR}.

\begin{definition}
\label{def:defect-cover}
Let \(E'\subseteq E(T)\), let \(P\subseteq U\), and let \(C\) be a matroid on \(P\).
For \(Z\subseteq P\), define the \emph{rank defect} of \(Z\) by
\[
  \defect_C(Z):=r_C(P)-r_C(Z).
\]
The triple \((E',P,C)\) has the \emph{defect cover property} if, for every \(Z\subseteq P\), there exist \(F\subseteq E'\) and a vertex cover \(D\) of \(A(F)-(P\setminus Z)\) such that
\[
  |D|\le q(F)-\defect_C(Z).
\]
\end{definition}

The defect cover property is the inductive tool that allows us to use [Hall] throughout the recursive construction.
Informally, it ensures that if the construction cannot accommodate the required representatives at some step, then some subfamily violates [Hall].

\begin{lemma}
\label{lem:reduction}
For each \(e\in E(T)\), let \(A'_e\) be obtained from \(A_e\) by deleting every nonloop hyperedge containing a resource that carries a loop in \(A_e\), and write \(\mathcal A':=(A'_e)_{e\in E(T)}\).
Then
\[
  \nu(A'(F))=\nu(A(F)) \qquad \text{for all } F\subseteq E(T),
\]
and \(\mathcal A\) admits a \(q\)-SDR if and only if \(\mathcal A'\) does.
\end{lemma}

\begin{proof}
In any matching, a deleted hyperedge \(\{x,y\}\in E(A_e)\), where \(x\) carries a loop in \(A_e\), may be replaced by that loop.
This preserves cardinality and disjointness; for a \(q\)-SDR, it also preserves the hypergraph \(A_e\) being represented.
Hence matching numbers and the existence of a \(q\)-SDR are unchanged.
\end{proof}

\begin{proof}[Proof of Theorem~\ref{thm:qSDR}]
Projecting each nonsingleton hyperedge \(\{z_{e,i}\}\cup a\) of a perfect matching of \(H(\mathcal A,q)\) to \(a\) gives a \(q\)-SDR.
Conversely, assigning the \(q_e\) selected representatives of \(A_e\) bijectively to \(z_{e,1},\ldots,z_{e,q_e}\) and adding singleton hyperedges on all unused resources gives a perfect matching of \(H(\mathcal A,q)\).
Thus \textup{[SDR]} and \textup{[Matching]} are equivalent, and it remains to prove that \textup{[SDR]} is equivalent to [Hall].

Necessity is immediate.
We prove sufficiency.

\subsubsection{Inductive Setup}
Fix a root \(\rho\in V(T)\).

By Lemma~\ref{lem:reduction}, we replace every \(A_e\) by \(A'_e\) for the remainder of the proof.

We now prove, by induction upward from the leaves of \(T\), that every \(\mathcal I_t\) is the family of independent sets of a reservation matroid \(M_t\) and that \((E_t,P_t,C_t)\) has the defect cover property.

If \(t\) is a leaf, then \(E_t=\emptyset\), so every subset of \(P_t\) is reservable.
Hence \(M_t\) is the free matroid on \(P_t\), while \(C_t\) has rank zero.
The defect cover property holds with \(F=D=\emptyset\).

Now let \(t\) be an internal vertex and assume that \(M_u\) and \(C_u\) have been constructed for every child \(u\) of \(t\).
In particular, Lemma~\ref{lem:semantic} applies to each \(C_u\).
For every child edge \(e=tu\), we first construct a reservation matroid for the subsystem indexed by \(\{e\}\cup E_u\), and then take the union of the dual consumption matroids of these subsystems.

\subsubsection{Construct the Reservation Matroid for a Child Edge}
We use a matroidal matching problem to encode which selections of \(q_e\) representatives from \(A_e\) can be combined with a \(q\)-SDR for \(E_u\).
Fix a child edge \(e=tu\), and write
\[
  L:=P_t,\qquad
  R:=P_u,\qquad
  S:=L\cap R,\qquad
  L_0:=L\setminus S,\qquad
  R_0:=R\setminus S.
\]
By Lemma~\ref{lem:reduction}, \(A_e\) consists of the loops on \(S\) and nonloop hyperedges between \(L_0\) and \(R_0\).

For each \(s\in S\), replace the loop at \(s\) by an edge \(s_Ls_R\), where \(s_L\) is placed in the left class and \(s_R\) in the right class.
Construct \(L\) as the union of \(L_0\) and all such vertices \(s_L\), and construct \(R\) as the union of \(R_0\) and all such vertices \(s_R\).
The resulting sets \(L\) and \(R\) are disjoint, and the nonloop hyperedges of \(A_e\) become edges joining \(L_0\) to \(R_0\).
This defines a bipartite graph \(K_e\) with bipartition \(L\sqcup R\).
Here \(s_L\) and \(s_R\) are distinct vertices of \(K_e\) representing the same resource \(s\); in set and matroid notation, we treat both as \(s\).

By the Rado--Perfect theorem \eqref{eq:rado-perfect-rank}, the subsets of \(L\) that can be matched to an \(M_u\)-independent subset of \(R\) form the independent sets of a matroid \(M_u[K_e]\) on \(L\), with
\begin{equation}
\label{eq:rado-rank}
  r_{M_u[K_e]}(H)=\min_{A\subseteq H}\left(|H\setminus A|+r_{M_u}(N_{K_e}(A))\right) \qquad \text{for all }H\subseteq L.
\end{equation}

Let \(N:=M_u[K_e]\).
Suppose for contradiction that \(r_N(L)<q_e\).

Choose \(A_L\subseteq L\) attaining \eqref{eq:rado-rank} for \(H=L\), and let \(Y_L:=N_{K_e}(A_L)\).
Then \((L\setminus A_L)\cup Y_L\) covers \(A_e\), and
\[
  r_N(L)=|L\setminus A_L|+r_{M_u}(Y_L)
  <q_e.
\]
Since \(C_u=M_u^*\), the dual rank identity \eqref{eq:dual-rank} gives
\[
  r_{C_u}(R)=|R|-r_{M_u}(R), \qquad
  r_{C_u}(R\setminus Y_L)=|R\setminus Y_L|-r_{M_u}(R)+r_{M_u}(Y_L).
\]
Therefore
\[
  \defect_{C_u}(R\setminus Y_L)=|Y_L|-r_{M_u}(Y_L).
\]
By the inductive hypothesis, \((E_u,P_u,C_u)\) has the defect cover property.
This gives \(F_u\subseteq E_u\) and a cover \(D_u\) of \(A(F_u)-Y_L\) satisfying \(|D_u|\le q(F_u)-|Y_L|+r_{M_u}(Y_L)\).
Consequently,
\[
  D:=(L\setminus A_L)\cup Y_L\cup D_u \qquad \text{covers \(A(F_u\cup\{e\})\)}
\]
and
\[
  |D|\le q(F_u)+|L\setminus A_L|+r_{M_u}(Y_L)
  =q(F_u)+r_N(L)
  <q(F_u)+q_e.
\]
Therefore
\[
  \nu\bigl(A(F_u\cup\{e\})\bigr)\le |D|
  <q(F_u\cup\{e\}),
\]
which contradicts the [Hall] condition.
Thus [Hall] forces \(r_N(L)\ge q_e\), so the local matching problem has enough rank to supply the required \(q_e\) representatives at \(e\).

Set \(k:=r_N(L)-q_e\ge 0\).
Let \(O_S\) be the rank zero matroid on \(S\), and let \(F_{L_0}\) be the free matroid on \(L_0\).
We now use \(N\) and the surplus \(k\) to define a candidate reservation matroid on \(L\).
\begin{equation}
\label{eq:branch-reservation}
  M_e:=\bigl((N^*|L_0)\oplus O_S\bigr)\vee\operatorname{Tr}_k\bigl((N|S)\oplus F_{L_0}\bigr), \qquad
  C_e:=M_e^*
\end{equation}
We claim that \(M_e\) is the reservation matroid for the subsystem indexed by \(\{e\}\cup E_u\).

Fix \(X\subseteq L\), and write
\[
  X_S:=X\cap S
  =X\cap R,
  \qquad X_0:=X\cap L_0.
\]
A \(q\)-SDR for \(\{e\}\cup E_u\) avoiding \(X\) is equivalently the union of \(q_e\) representatives from \(A_e\) and a \(q\)-SDR for \(E_u\), where the two selections are pairwise vertex disjoint and neither consumes a resource of \(X\).
After splitting the loops on \(S\), the \(q_e\) representatives from \(A_e\) correspond to a matching \(J\) of size \(q_e\) in \(K_e\).
They avoid \(X\) exactly when the left endpoints of \(J\) lie in \(L\setminus X\).
Let \(Y\subseteq R\) be the set of right endpoints of \(J\).
The \(q\)-SDR for \(E_u\) must then avoid \(Y\), whose resources are consumed by the representatives from \(A_e\), and \(X_S=X\cap R\).
Among the resources of \(X\), only those in \(X_S\) can be consumed by the subsystem indexed by \(E_u\): if \(x\in X\) can be consumed below \(u\), then \(t\in V(Q_x)\), and some vertex of \(T_u\) also belongs to \(V(Q_x)\).
The path between these vertices passes through \(u\), so the subtree \(Q_x\) contains this path and hence \(u\in V(Q_x)\).
Thus \(x\in P_u=R\), and hence \(x\in X\cap R=X_S\).
By the definition of \(M_u\), a \(q\)-SDR for \(\{e\}\cup E_u\) avoiding \(X\) exists if and only if \(K_e\) has a matching of size \(q_e\), with left endpoints in \(L\setminus X\) and right endpoint set \(Y\subseteq R\), such that \(X_S\cup Y\) is independent in \(M_u\).

For such a matching \(J\),
\[
  J\cup\{s_Ls_R:s\in X_S\}
\]
is a matching from \(|X_S|+q_e\) left vertices, including \(X_S\), to an \(M_u\)-independent subset of \(R\).
Conversely, any such matching containing \(X_S\) on the left must use the identity edges at \(X_S\), and deleting them recovers a matching \(J\) as above.
By the definition of \(N=M_u[K_e]\), the existence of \(J\) is therefore equivalent to \(X_S\) being independent in \(N\) and extendable within \(L\setminus X_0\) to an independent set of size \(|X_S|+q_e\).

Since \(X_S\subseteq L\setminus X_0\), whenever \(X_S\) is independent in \(N\), it is independent in \(N|(L\setminus X_0)\).
By the matroid extension property, \(X_S\) extends to a basis of this restriction.
Consequently, it extends to an independent set of size \(|X_S|+q_e\) when
\[
  r_N(L\setminus X_0)\ge |X_S|+q_e.
\]
The dual rank identity \eqref{eq:dual-rank} and \(r_N(L)=q_e+k\) now show that such a matching \(J\) exists when
\begin{equation}
\label{eq:reservation-conditions}
  X_S\text{ is independent in }N,
  \qquad |X|-r_{N^*}(X_0)\le k.
\end{equation}
To verify that these are the exact conditions for independence in \eqref{eq:branch-reservation}, write its two summands as
\[
  M^{(1)}:=(N^*|L_0)\oplus O_S,
  \qquad M^{(2)}:=\operatorname{Tr}_k\bigl((N|S)\oplus F_{L_0}\bigr).
\]
Thus \(M_e=M^{(1)}\vee M^{(2)}\), and \(X\) is independent in \(M_e\) when it can be partitioned as \(X=I_1\sqcup I_2\) with \(I_i\) independent in \(M^{(i)}\) for \(i \in \{1,2\}\).
Since \(O_S\) has rank zero, \(I_1\) contains no element of \(X_S\); hence \(X_S\subseteq I_2\), and the restriction \(N|S\) in \(M^{(2)}\) requires \(X_S\) to be independent in \(N\).
Moreover, \(I_1\) contains at most \(r_{N^*}(X_0)\) elements, whereas the truncation in \(M^{(2)}\) gives \(|I_2|\le k\).
Consequently, \(|X|-r_{N^*}(X_0)\le |I_2|\le k\).
This proves the necessity of \eqref{eq:reservation-conditions}.

Conversely, suppose that \eqref{eq:reservation-conditions} holds.
Choose a basis \(I_1\) of \(N^*|X_0\) and set \(I_2:=X\setminus I_1\).
Then \(I_1\) is independent in \(M^{(1)}\), while the \(S\)-part of \(I_2\) is \(X_S\), its \(L_0\)-part is independent in the free matroid \(F_{L_0}\), and \(|I_2|=|X|-r_{N^*}(X_0)\le k\).
Thus \(I_2\) is independent in \(M^{(2)}\), and hence \(X\) is independent in \(M_e\).
Thus \(M_e\) is the required reservation matroid.

\subsubsection{Pass the Defect Cover Property through \(e\)}
We next prove that \((\{e\}\cup E_u,L,C_e)\) has the defect cover property.

Fix \(Z\subseteq L\), and let \(X:=L\setminus Z\).
The dual rank identity \eqref{eq:dual-rank} gives
\begin{equation}
\label{eq:branch-defect}
  \begin{aligned}
  \defect_{C_e}(Z)&=\bigl(|L|-r_{M_e}(L)\bigr)-\bigl(|Z|-r_{M_e}(L)+r_{M_e}(X)\bigr)\\
  &=|X|-r_{M_e}(X).
  \end{aligned}
\end{equation}
Apply the finite matroid union rank formula \eqref{eq:matroid-union-rank} to the representation of \(M_e\) in \eqref{eq:branch-reservation}, and choose \(W\subseteq X\) attaining the minimum.
Write
\[
  W_0:=W\cap L_0,\qquad
  W_S:=W\cap S.
\]
The ranks of the two summands in \eqref{eq:branch-reservation} on \(W\) are
\[
  \begin{aligned}
  r_{M^{(1)}}(W)&=r_{N^*|L_0}(W_0)+r_{O_S}(W_S)
  =r_{N^*}(W_0),\\
  r_{M^{(2)}}(W)&=\min\bigl\{k,r_{N|S}(W_S)+r_{F_{L_0}}(W_0)\bigr\}
  =\min\{k,r_N(W_S)+|W_0|\}.
  \end{aligned}
\]
Substituting the resulting expression for \(r_{M_e}(X)\) into \eqref{eq:branch-defect} gives
\begin{equation}
\label{eq:branch-union-defect}
  \begin{aligned}
  \defect_{C_e}(Z)&=|X|-\Bigl(|X\setminus W|+r_{N^*}(W_0)+\min\{k,r_N(W_S)+|W_0|\}\Bigr)\\
  &=|W|-r_{N^*}(W_0)-\min\{k,r_N(W_S)+|W_0|\}.
\end{aligned}
\end{equation}

The minimum in \eqref{eq:branch-union-defect} leads to two cases, according to whether the truncation at rank \(k\) is active.

Suppose first that \(r_N(W_S)+|W_0|\le k\).
Since \(W=W_0\sqcup W_S\), \(N|S=M_u|S\), and the dual rank identity \eqref{eq:dual-rank} give
\[
  \begin{aligned}
  \defect_{C_e}(Z)&=|W|-r_{N^*}(W_0)-r_N(W_S)-|W_0|
  =|W_S|-r_N(W_S)-r_{N^*}(W_0)\\
  &\le|W_S|-r_N(W_S)
  =|W_S|-r_{M_u}(W_S)
  =\defect_{C_u}(R\setminus W_S).
  \end{aligned}
\]
The inequality uses \(r_{N^*}(W_0)\ge0\).
Apply the defect cover property for \((E_u,R,C_u)\) at \(R\setminus W_S\).
It gives \(F_u\subseteq E_u\) and a cover \(D\) of \(A(F_u)-W_S\) such that
\[
  |D|\le q(F_u)-\defect_{C_u}(R\setminus W_S)
  \le q(F_u)-\defect_{C_e}(Z).
\]
Since \(W_S\subseteq X = L \setminus Z\), and resources in \(L_0\) do not occur in the subsystem indexed by \(E_u\), the same set covers \(A(F_u)-X\).
Thus, with \(F:=F_u\subseteq\{e\}\cup E_u\), the set \(D\) satisfies both requirements of the defect cover property for \((\{e\}\cup E_u,L,C_e)\) at \(Z\).

Now suppose that \(r_N(W_S)+|W_0|>k\).
Let \(H:=L\setminus W_0\).
Using the dual rank identity \eqref{eq:dual-rank} and \eqref{eq:branch-union-defect} gives
\begin{equation}
\label{eq:truncated-defect}
  \defect_{C_e}(Z)=q_e+|W_S|-r_N(H).
\end{equation}
Choose \(A_H\subseteq H\) attaining \eqref{eq:rado-rank} for \(H\), and let \(Y_H:=N_{K_e}(A_H)\).
Thus
\begin{equation}
\label{eq:rado-witness}
  r_N(H)=|H\setminus A_H|+r_{M_u}(Y_H).
\end{equation}
After deleting \(X\), every remaining left endpoint lies in \(L\setminus X=Z\).
Since \(W_0\subseteq X\), we also have \(Z\subseteq L\setminus W_0=H\).
If it does not belong to \(A_H\), it lies in \(Z\setminus A_H\); if it does, its right endpoint lies in \(Y_H=N_{K_e}(A_H)\), and that endpoint either is deleted with \(X_S\) or lies in \(Y_H\setminus X_S\).
Consequently,
\[
  (Z\setminus A_H)\cup(Y_H\setminus X_S) \qquad \text{covers \(A_e-X\)}.
\]

Let \(Q:=Y_H\cup X_S\).
The defect cover property for \((E_u,R,C_u)\), applied with \(R\setminus Q\), gives \(F_u\subseteq E_u\) and a cover \(D_u\) of \(A(F_u)-Q\) such that
\[
  |D_u|\le q(F_u)-|Q|+r_{M_u}(Q).
\]
Therefore
\[
  D:=(Z\setminus A_H)\cup(Y_H\setminus X_S)\cup D_u \qquad \text{covers \(A(F_u\cup\{e\})-X\)}.
\]
Indeed, the first two terms cover \(A_e-X\), as shown above.
Every hyperedge of \(A(F_u)-X\) that avoids \(Q\) is covered by \(D_u\); any remaining hyperedge meets \(Q\setminus X=Y_H\setminus X_S\), which is contained in \(D\).

Using the bound on \(D_u\), \(|Y_H\setminus X_S|-|Y_H\cup X_S|=-|X_S|\), and \(r_{M_u}(Y_H\cup X_S) \le r_{M_u}(Y_H)+|X_S\setminus Y_H|\), we obtain
\begin{equation}
\label{eq:branch-cover-bound}
\begin{aligned}
  |D|&\le|Z\setminus A_H|+|Y_H\setminus X_S|+|D_u|\\
  &\le q(F_u)+|Z\setminus A_H|-|X_S|+r_{M_u}(Y_H\cup X_S)\\
  &\le q(F_u)+|Z\setminus A_H|+r_{M_u}(Y_H)-|X_S\cap Y_H|.
\end{aligned}
\end{equation}
For each \(s\in W_S\), consider its left copy \(s_L\).
If \(s_L\notin (X\setminus W_0)\setminus A_H\), then the identity edge \(s_Ls_R\) implies that \(s_R\in Y_H\); since \(s\in W_S\subseteq X_S\), this resource is counted in \(X_S\cap Y_H\).
Therefore
\[
  |(X\setminus W_0)\setminus A_H|+|X_S\cap Y_H|\ge |W_S|.
\]
Since \(H\setminus A_H\) is the disjoint union of \(Z\setminus A_H\) and \((X\setminus W_0)\setminus A_H\), \eqref{eq:branch-cover-bound}, \eqref{eq:rado-witness}, and \eqref{eq:truncated-defect} give
\[
  \begin{aligned}
  |D|&\le q(F_u)+|Z\setminus A_H|+r_{M_u}(Y_H)+|(X\setminus W_0)\setminus A_H|-|W_S|\\
  &=q(F_u)+r_N(H)-|W_S|
  =q(F_u)+q_e-\defect_{C_e}(Z)\\
  &=q(F_u\cup\{e\})-\defect_{C_e}(Z).
  \end{aligned}
\]
Since \(D\) covers \(A(F_u\cup\{e\})-X\) and \(X=L\setminus Z\), taking \(F:=F_u\cup\{e\}\) proves the defect cover property for \((\{e\}\cup E_u,L,C_e)\) at \(Z\).

\subsubsection{Combine the Child Subsystems}
Let \(e_i=tu_i\), \(1\le i\le s\), be the child edges of \(t\), set \(C_i:=C_{e_i}\), and define
\[
  C:=C_1\vee\cdots\vee C_s.
\]
Fix \(Z\subseteq P_t\).
By the finite matroid union rank formula \eqref{eq:matroid-union-rank}, choose \(W\subseteq Z\) such that
\[
  r_C(Z)=|Z\setminus W|+\sum_{i=1}^{s}r_{C_i}(W).
\]
For each \(i\), the previous argument shows that \((\{e_i\}\cup E_{u_i},P_t,C_i)\) has the defect cover property.
Applying this property at \(W\) gives \(F_i\subseteq\{e_i\}\cup E_{u_i}\) and a cover \(D_i\) of \(A(F_i)-(P_t\setminus W)\) satisfying
\[
  |D_i|\le q(F_i)-\bigl(r_{C_i}(P_t)-r_{C_i}(W)\bigr).
\]
Let
\[
  F:=\bigcup_{i=1}^{s}F_i,
  \qquad D:=(Z\setminus W)\cup\bigcup_{i=1}^{s}D_i.
\]
Then \(D\) covers \(A(F)-(P_t\setminus Z)\), and
\begin{equation}
\label{eq:combined-cover-estimate}
\begin{aligned}
  |D|&\le q(F)-\sum_{i=1}^{s}\bigl(r_{C_i}(P_t)-r_{C_i}(W)\bigr)+|Z\setminus W|\\
  &=q(F)-\left(\sum_{i=1}^{s}r_{C_i}(P_t)-r_C(Z)\right).
\end{aligned}
\end{equation}
Taking \(Z=P_t\) in \eqref{eq:combined-cover-estimate}, the parenthetical term becomes the rank lost when the child consumption matroids are combined.
Note that the set \(D\) covers \(A(F)\).
If \(r_C(P_t)<\sum_{i=1}^{s}r_{C_i}(P_t)\), then \eqref{eq:combined-cover-estimate} gives
\[
  \nu(A(F))\le |D|
  <q(F),
\]
contradicting the [Hall] condition.
Therefore
\begin{equation}
\label{eq:rank-additivity}
  r_C(P_t)=\sum_{i=1}^{s}r_{C_i}(P_t).
\end{equation}
Thus the child consumption matroids combine without rank loss.
Define
\[
  M_t:=C^*,
  \qquad C_t:=C.
\]
Since \(Z\subseteq P_t\) was arbitrary, equations \eqref{eq:combined-cover-estimate} and \eqref{eq:rank-additivity} give
\[
  |D|\le q(F)-\bigl(r_{C_t}(P_t)-r_{C_t}(Z)\bigr)
  =q(F)-\defect_{C_t}(Z).
\]
Thus \((E_t,P_t,C_t)\) has the defect cover property.
Every basis \(B\) of \(C\) decomposes as
\[
  B=B_1\sqcup \cdots \sqcup B_s, \qquad \text{where each \(B_i\) is a basis of \(C_i\)}.
\]
By Lemma~\ref{lem:semantic}, for each \(i\), there is a \(q\)-SDR for the corresponding subsystem that consumes exactly \(B_i\) among the resources of \(P_t\).
Suppose that two child \(q\)-SDRs consume the same resource \(x\).
If \(x\notin P_t\), then \(Q_x\) meets two distinct child subtrees; since \(Q_x\) is connected, this forces \(t\in V(Q_x)\), a contradiction.
If \(x\in P_t\), the corresponding sets \(B_i\) and \(B_j\), where \(i\ne j\), both contain \(x\), contradicting the disjointness of the \(B_\ell\).
Therefore the child \(q\)-SDRs are pairwise vertex disjoint, and their union is a \(q\)-SDR for \(E_t\).

\subsubsection{Identification of \(M_t\) and completion at the root}
We verify that \(M_t\) represents \(\mathcal I_t\).
If \(X\) is independent in \(M_t\), extend it to a basis \(X'\) of \(M_t\).
Then \(P_t\setminus X'\) is a basis of \(C_t\) disjoint from \(X\).
Decompose it into bases of the \(C_i\) and realize them by \(q\)-SDRs as above.
Their union consumes no resource of \(X\), so \(X\in\mathcal I_t\).

Conversely, suppose that \(X\in\mathcal I_t\), and fix a \(q\)-SDR for \(E_t\) consuming no resource of \(X\).
Its restriction to each child subsystem consumes some set \(Y_i\subseteq P_t\).
Choose an inclusion-minimal subset \(B_i\subseteq Y_i\) consumed by a \(q\)-SDR for that subsystem.
Lemma~\ref{lem:semantic} implies that \(B_i\) is a basis of \(C_i\).
The \(B_i\) are pairwise disjoint, so \(B:=B_1\sqcup \cdots \sqcup B_s\) is a basis of \(C\).
Since \(B\cap X=\emptyset\), the basis \(P_t\setminus B\) of \(M_t\) contains \(X\).
Hence \(X\) is independent in \(M_t\), and therefore \(\mathcal I_t\) is the independent set family of \(M_t\).
This completes the induction.

At the root \(\rho\), the empty set is independent in \(M_\rho\).
Hence, by the definition of \(\mathcal I_\rho\), there is a \(q\)-SDR for \(E_\rho=E(T)\).
This proves the theorem.
\end{proof}

\subsection{Sharpness of Acyclicity}
We now prove the classification stated in Theorem~\ref{thm:forest-restriction}.
Consequently, extensions to unicyclic graphs, cacti, graphs of treewidth at most \(k\) for any fixed \(k\ge2\), or chordal graphs require additional structure, either through stronger hypotheses or through inequalities recording the states around cycles.

\begin{proof}[Proof of Theorem~\ref{thm:forest-restriction}]
Suppose first that \(G\) is a forest.
Every connected resource support lies in one component of \(G\).
For each component containing an edge, restrict to the resources whose supports lie in that component and apply Theorem~\ref{thm:qSDR}.
The resulting \(q\)-SDRs use disjoint resource sets, so their union is a \(q\)-SDR for the system on \(G\).

Conversely, suppose that \(G\) contains a cycle.
Choose consecutive cycle edges \(e_1=uv, \ e_2=vw\) and let \(R\) be the \(u\)--\(w\) path around the remainder of the cycle, so \(v\notin V(R)\).
Introduce resources \(a,b,c,d\) with
\[
  Q_a=\{u\}, \qquad
  Q_b=Q_d
  =\{v\}, \qquad
  Q_c=R,
\]
and let
\[
  E(A_{e_1})=\bigl\{\{a,b\},\{c,d\}\bigr\}, \qquad
  E(A_{e_2})=\bigl\{\{b,c\}\bigr\}.
\]

For each \(z=xy\in E(G)\setminus\{e_1,e_2\}\), introduce a distinct resource \(p_z\) and set \(Q_{p_z}=z\).
Choose \(A_z\) to contain only the loops prescribed by condition~(i) of Definition~\ref{def:tree-system}; explicitly,
\[
  E(A_z)=
  \begin{cases}
    \bigl\{\{p_z\},\{c\}\bigr\}, & x,y\in V(R),\\
    \bigl\{\{p_z\}\bigr\},       & \text{otherwise}.
  \end{cases}
\]
Set \(q(e)=1\) for every \(e\in E(G)\).

Fix \(F\subseteq E(G)\).
For every \(z\in F\setminus\{e_1,e_2\}\), select the loop at \(p_z\).
If \(F\cap\{e_1,e_2\}=\{e_1\}\), add \(\{a,b\}\), and if \(F\cap\{e_1,e_2\}=\{e_2\}\), add \(\{b,c\}\).
If \(F\) contains both \(e_1\) and \(e_2\), add \(\{a,b\}\) and \(\{c,d\}\).
The selected hyperedges are pairwise vertex disjoint, and hence
\[
  \nu(A(F))\ge |F|
  =q(F) \qquad \text{for all } F\subseteq E(G).
\]

The sole representative \(\{b,c\}\) of \(A_{e_2}\) meets both representatives \(\{a,b\}\) and \(\{c,d\}\) of \(A_{e_1}\), hence there is no \(q\)-SDR.
Therefore the [Hall] condition is not sufficient for any graph containing a cycle.
\end{proof}

\section{\texorpdfstring{Hamilton-connectedness of \(5\)-tough chordal graphs}{Hamilton-connectedness of 5-tough chordal graphs}}
\label{sec:five-toughness}

A graph is \emph{chordal} if every cycle of length at least four has a chord.
Chv\'atal introduced \emph{toughness} in 1973 in connection with Hamilton cycles~\cite{Chvatal}.
For a noncomplete graph \(G\), define
\[
  t(G):=\min\left\{\frac{|S|}{c(G-S)}:S\subseteq V(G),\ c(G-S)\ge2\right\},
\]
where \(c(H)\) denotes the number of components of \(H\).
A set \(S\subseteq V(G)\) is \emph{separating} if \(c(G-S)\ge2\).
For complete graphs, set \(t(K_n):=\infty\).
A graph \(G\) is \emph{\(t\)-tough} if
\[
  |S|\ge t\,c(G-S) \qquad \text{whenever } S\subseteq V(G)\text{ and }c(G-S)\ge2.
\]
We note that every Hamiltonian graph is \(1\)-tough, since deleting \(S\) from a Hamilton cycle leaves at most \(|S|\) path components.

Our proof of Theorem~\ref{thm:five-tough-hamilton-connected} builds on the framework of Kabela and Kaiser~\cite{KabelaKaiser}, who reduce Hamiltonicity of chordal graphs to the existence of an SDR.
We use the following reformulations of two lemmas\footnote{The first follows from \cite[Lemma~9]{KabelaKaiser} by taking \(k=2/5\) and multiplying by \(5\); the second is \cite[Lemma~10]{KabelaKaiser}, with red edges renamed gray and \(e\)-enclosing pairs expressed in terms of the supports \(Q_x\).} in our notation.

\begin{lemma}
\label{lem:kk-tree-estimate}
Let \(T\) be a tree and \(B\subseteq E(T)\).
For \(i\in\{0,1,2\}\), let \(b_i\) be the number of edges of \(B\) having exactly \(i\) endpoints of degree at most two in \(T\).
If \(r\) is the number of components of \(T-B\) containing a vertex of degree at most two in \(T\), then
\[
  5(r-1)\ge2b_0+3b_1+5b_2.
\]
\end{lemma}

\begin{lemma}
\label{lem:kk-separation}
Use the notation of Subsection~\ref{subsec:chordal-system}.
Let \(e\in E(T)\), let \(C\subseteq U\) be a vertex cover of \(A_e\), and suppose that \(Q_x\) and \(Q_y\) contain vertices in different components of \(T-e\).
If \(e\) is black and \(x,y\notin C\), then \(x\) and \(y\) lie in different components of \(G-C\).
If \(e\) is gray and \(x,y\notin C\cup\{i_e\}\), then \(x\) and \(y\) lie in different components of \(G-(C\cup\{i_e\})\).
\end{lemma}

We associate a tree-structured hypergraph system with a subtree representation of \(G\).
A \(q\)-SDR for this system can be read along a spanning walk of \(T\) to produce a Hamilton path (Lemma~\ref{lem:endpoint-construction}).
If the Hall condition fails, a violated inequality yields a small vertex cover, which Lemma~\ref{lem:cover-to-separator} converts into a separator contradicting \(5\)-toughness.

\subsection{The Hypergraph System Built from a Chordal Graph}
\label{subsec:chordal-system}
Let \(G\) be a chordal graph.
By Gavril~\cite{Gavril}, there exist a tree \(T_0\) and a family \(\mathcal F=(F_x)_{x\in V(G)}\) of its subtrees such that, for distinct \(x,y\in V(G)\),
\[
  xy\in E(G)\quad\Longleftrightarrow\quad V(F_x)\cap V(F_y)\ne\emptyset.
\]

Choose this representation so that \(|V(T_0)|\) is minimum.
Next choose an inclusion-maximal independent set \(I\) such that, for every \(i\in I\),
\begin{enumerate}
  \item[(i)] \(F_i\) is a path whose vertices have degree at most two in \(T_0\);
  \item[(ii)] \(F_i\) does not properly contain any represented subtree.
\end{enumerate}
Such a choice is possible: minimality of \(T_0\) implies that every leaf \(\ell\) occurs as some \(F_x=\{\ell\}\), for otherwise deleting \(\ell\) would preserve the representation.
Such an \(F_x\) satisfies (i) and (ii).

Suppress every degree-two vertex of \(T_0\) that is not an endpoint of any \(F_i\), \(i\in I\), and denote the resulting tree by \(T\).

Every represented subtree contains a vertex of \(T\)~\cite[Proposition~3]{KabelaKaiser}.
For every \(x\in V(G)\), define
\[
  Q_x:=T[V(F_x)\cap V(T)].
\]
Then \(Q_x\) is a nonempty subtree of \(T\).

Since \(I\) is independent, the subtrees \(F_i\), \(i\in I\), are pairwise vertex disjoint.
By (i), every interior vertex of \(F_i\) is therefore suppressed, so \(Q_i\) is either a single vertex or a single edge of \(T\).
Color an edge \(e\) of \(T\) gray if \(e=Q_i\) for some \(i\in I\), and black otherwise.
Such an \(i\) is unique; write \(i_e:=i\).
Both endpoints of every gray edge have degree two in \(T\).
Indeed, they have degree at most two by (i), while a leaf endpoint would yield a represented singleton properly contained in \(F_{i_e}\), contradicting (ii).

Set \(U:=V(G)\setminus I\), and define
\[
  P_t:=\{x\in U:t\in V(Q_x)\} \qquad \text{for all } t\in V(T).
\]
For \(e=tu\in E(T)\), define a hypergraph \(A_e\) on \(U\).
Its loops are the singleton hyperedges \(\{x\}\) with \(x\in P_t\cap P_u\), and its nonloop hyperedges are the pairs \(\{x,y\}\) such that \(xy\in E(G)\) and, after possibly exchanging \(x,y\), we have \(x\in P_t\) and \(y\in P_u\).
Thus \((A_e)_{e\in E(T)}\) is a tree-structured hypergraph system in the sense of Definition~\ref{def:tree-system}.

Define the multiplicities by
\begin{equation}
\label{eq:chordal-multiplicity}
  q_e:=
  \begin{cases}
    2,&e\text{ is black},\\
    1,&e\text{ is gray}.
  \end{cases}
\end{equation}

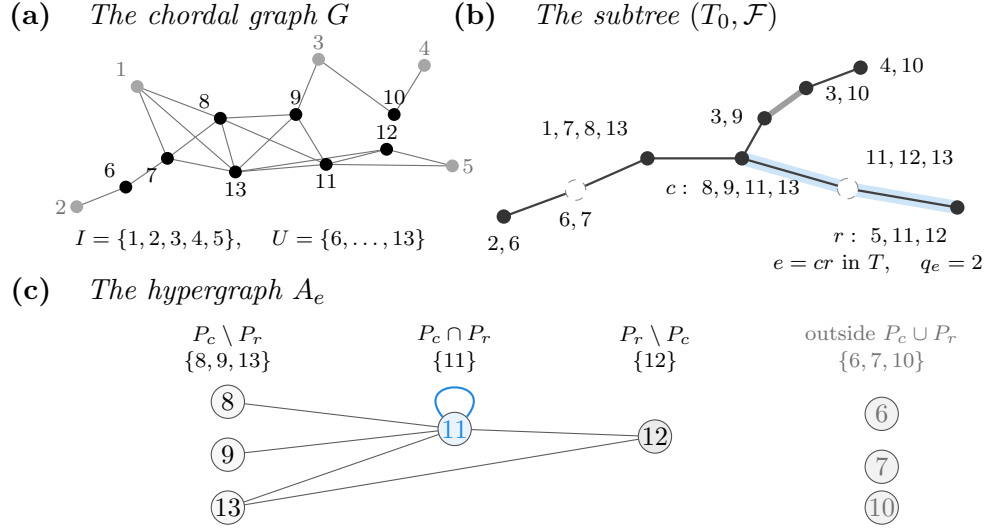
\begin{figure}[ht]
\centering
\begin{tikzpicture}[
  gvertex/.style={circle,fill=black,minimum size=1.7mm,inner sep=0pt},
  ivertex/.style={circle,fill=black!35,minimum size=1.7mm,inner sep=0pt},
  hvertex/.style={
    circle,draw=black!70,fill=white,minimum size=4.4mm,inner sep=0pt
  },
  treevertex/.style={
    circle,fill=black!80,minimum size=1.9mm,inner sep=0pt
  },
  suppressed/.style={
    circle,draw=black!45,dashed,fill=white,minimum size=2.7mm,inner sep=0pt
  },
  every node/.style={font=\small}
]

  \node[anchor=east,font=\bfseries] at (-5.85,2.6) {(a)};
  \node[anchor=west,font=\itshape] at (-5.65,2.6)
    {The chordal graph \(G\)};

  \coordinate (g1)  at (-4.85,1.67);
  \coordinate (g2)  at (-5.65,.08);
  \coordinate (g3)  at (-2.45,2.03);
  \coordinate (g4)  at (-1.05,1.95);
  \coordinate (g5)  at (-.68,.62);
  \coordinate (g6)  at (-5.00,.34);
  \coordinate (g7)  at (-4.45,.72);
  \coordinate (g8)  at (-3.75,1.25);
  \coordinate (g9)  at (-2.75,1.30);
  \coordinate (g10) at (-1.45,1.30);
  \coordinate (g11) at (-2.35,.64);
  \coordinate (g12) at (-1.55,.84);
  \coordinate (g13) at (-3.55,.54);

  \foreach \x/\y in {
    1/7,1/8,1/13,
    2/6,
    3/9,3/10,
    4/10,
    5/11,5/12,
    6/7,
    7/8,7/13,
    8/9,8/11,8/13,
    9/11,9/13,
    11/12,11/13,
    12/13}
    \draw[black!55,line width=.42pt] (g\x)--(g\y);

  \foreach \x in {1,2,3,4,5}
    \node[ivertex] at (g\x) {};
  \foreach \x in {6,7,8,9,10,11,12,13}
    \node[gvertex] at (g\x) {};

  \node[font=\scriptsize,text=black!55,anchor=south east] at (g1) {\(1\)};
  \node[font=\scriptsize,text=black!55,anchor=east]       at (g2) {\(2\)};
  \node[font=\scriptsize,text=black!55,anchor=south]      at (g3) {\(3\)};
  \node[font=\scriptsize,text=black!55,anchor=south]      at (g4) {\(4\)};
  \node[font=\scriptsize,text=black!55,anchor=west]       at (g5) {\(5\)};
  \node[font=\scriptsize,anchor=south east] at (g6)  {\(6\)};
  \node[font=\scriptsize,anchor=north east] at (g7)  {\(7\)};
  \node[font=\scriptsize,anchor=south east] at (g8)  {\(8\)};
  \node[font=\scriptsize,anchor=south]      at (g9)  {\(9\)};
  \node[font=\scriptsize,anchor=south]      at (g10) {\(10\)};
  \node[font=\scriptsize,anchor=north]      at (g11) {\(11\)};
  \node[font=\scriptsize,anchor=south]      at (g12) {\(12\)};
  \node[font=\scriptsize,anchor=north]      at (g13) {\(13\)};
  \node[font=\scriptsize] at (-3.35,-.35)
    {\(I=\{1,2,3,4,5\}\), \quad \(U=\{6,\ldots,13\}\)};

  \node[anchor=east,font=\bfseries] at (.05,2.6) {(b)};
  \node[anchor=west,font=\itshape] at (.25,2.6)
    {The subtree \((T_0,\mathcal F)\)};

  \coordinate (tlthree) at (.00,-.05);
  \coordinate (tltwo)   at (.95,.30);
  \coordinate (tlone)   at (1.90,.72);
  \coordinate (tc)      at (3.15,.72);
  \coordinate (tuone)   at (3.45,1.25);
  \coordinate (tutwo)   at (4.00,1.65);
  \coordinate (tuthree) at (4.72,1.92);
  \coordinate (tm)      at (4.55,.32);
  \coordinate (tr)      at (6.00,.07);

  \draw[babyblue,line width=5pt,opacity=.22,line cap=round]
    (tc)--(tm)--(tr);
  \draw[black!75,line width=.85pt]
    (tlthree)--(tltwo)--(tlone)--(tc)--(tuone)
    (tutwo)--(tuthree)
    (tc)--(tm)--(tr);
  \draw[black!38,line width=2.2pt]
    (tuone)--(tutwo);

  \foreach \x in {tlthree,tlone,tc,tuone,tutwo,tuthree,tr}
    \node[treevertex] at (\x) {};
  \foreach \x in {tltwo,tm}
    \node[suppressed] at (\x) {};

  \node[font=\scriptsize,anchor=north] at (.00,-.18) {\(2,6\)};
  \node[font=\scriptsize,anchor=north] at (.95,.14) {\(6,7\)};
  \node[font=\scriptsize,anchor=south east] at (1.78,.84) {\(1,7,8,13\)};
  \node[font=\scriptsize,anchor=north] at (3.03,.57) {\(c:\ 8,9,11,13\)};
  \node[font=\scriptsize,anchor=east] at (3.30,1.25) {\(3,9\)};
  \node[font=\scriptsize,anchor=west] at (4.12,1.55) {\(3,10\)};
  \node[font=\scriptsize,anchor=west] at (4.84,1.92) {\(4,10\)};
  \node[font=\scriptsize,anchor=south west] at (4.65,.46) {\(11,12,13\)};
  \node[font=\scriptsize,anchor=north east] at (6.00,-.07) {\(r:\ 5,11,12\)};

  \node[font=\scriptsize] at (4.95,-.67)
    {\(e=cr\) in \(T\), \quad \(q_e=2\)};

  \node[anchor=east,font=\bfseries] at (-5.85,-1.05) {(c)};
  \node[anchor=west,font=\itshape] at (-5.65,-1.05)
    {The hypergraph \(A_e\)};

  \coordinate (a8)  at (-3.65,-2.50);
  \coordinate (a9)  at (-3.65,-3.20);
  \coordinate (a13) at (-3.65,-3.90);
  \coordinate (a11) at (-.65,-2.86);
  \coordinate (a12) at (2.00,-2.96);
  \coordinate (a6)  at (5.00,-2.66);
  \coordinate (a7)  at (5.00,-3.36);
  \coordinate (a10) at (5.00,-3.90);

  \draw[black!65,thin]
    (a11)--(a8) (a11)--(a9) (a11)--(a13)
    (a11)--(a12) (a13)--(a12);

  \node[hvertex,fill=black!3,text=black] at (a8)  {\(8\)};
  \node[hvertex,fill=black!3,text=black] at (a9)  {\(9\)};
  \node[hvertex,fill=black!3,text=black] at (a13) {\(13\)};
  \node[hvertex,fill=babyblue!10,text=babyblue] (aeleven) at (a11) {\(11\)};
  \node[hvertex,fill=black!8,text=black] at (a12) {\(12\)};
  \node[hvertex,fill=black!5,text=black!55] at (a6)  {\(6\)};
  \node[hvertex,fill=black!5,text=black!55] at (a7)  {\(7\)};
  \node[hvertex,fill=black!5,text=black!55] at (a10) {\(10\)};

  \draw[babyblue,thick]
    (aeleven.45) .. controls +(4.2mm,5.2mm) and +(-4.2mm,5.2mm) ..
    (aeleven.135);

  \node[font=\scriptsize,align=center] at (-3.65,-1.80)
    {\(P_c\setminus P_r\)\\\(\{8,9,13\}\)};
  \node[font=\scriptsize,align=center] at (-.65,-1.80)
    {\(P_c\cap P_r\)\\\(\{11\}\)};
  \node[font=\scriptsize,align=center] at (2.00,-1.80)
    {\(P_r\setminus P_c\)\\\(\{12\}\)};
  \node[font=\scriptsize,align=center,text=black!55] at (5.00,-1.80)
    {outside \(P_c\cup P_r\)\\\(\{6,7,10\}\)};
\end{tikzpicture}
\caption{An example of \(13\) vertices.
In (b), a node \(v\) is labeled by \(\{x:v\in V(F_x)\}\); the hollow degree-two vertices are suppressed.}
\label{fig:chordal-example}
\end{figure}

\subsection{\texorpdfstring{\(q\)-SDRs imply Hamilton paths}{q-SDRs imply Hamilton paths}}

\begin{lemma}
\label{lem:endpoint-construction}
Let \(u,v\) be distinct vertices of \(G\), let
\[
  Z:=\{u,v\}\setminus I.
\]
If the tree-structured hypergraph system \((A_e-Z)_{e\in E(T)}\) admits a \(q\)-SDR, then \(G\) has a Hamilton path from \(u\) to \(v\).
\end{lemma}

\begin{proof}
If \(T\) has no edge, every represented subtree contains the unique vertex of \(T\), so \(G\) is complete and the conclusion is immediate.
Assume \(E(T)\ne\emptyset\).
Choose
\[
  s\in V(F_u)\cap V(T),
  \qquad t\in V(F_v)\cap V(T)
\]
so that the unique \(s\)--\(t\) path in \(T\), denoted by \(P\), contains every gray edge \(e\) for which \(i_e\in\{u,v\}\), if any.
Starting at \(s\), follow \(P\), making a round trip through every branch off \(P\) before continuing.
Let \(W\) be the resulting \(s\)--\(t\) walk in \(T\).
It visits every vertex of \(T\), uses each edge of \(P\) once, and uses every other edge twice.

For each \(e\in E(T)\), let \(M_e\subseteq E(A_e-Z)\) be the set of \(q_e\) hyperedges selected by the \(q\)-SDR.
Assign them to the occurrences of \(e\) in \(W\) as follows.
Use the two members of \(M_e\) on the two occurrences of a black edge outside \(P\), and either member on a black edge in \(P\).
Use the unique member of \(M_e\) on a gray edge in \(P\); if \(e\notin E(P)\) is gray, use this member on one occurrence and \(\{i_e\}\) on the other.

Prepend \(u\), then read \(W\) from \(s\) to \(t\).
For an assigned singleton \(\{x\}\), record \(x\); for an assigned pair \(\{x,y\}\), record \(x,y\) in the direction of the walk.
Finally, append \(v\).

The recorded vertices form a path in \(G\): vertices from the same pair are adjacent by the definition of \(A_e\); if two consecutive edges of \(W\) meet at \(r\in V(T)\), the recording for the first edge ends with some \(x\) such that \(r\in V(Q_x)\), and the recording for the second edge begins with some \(y\) such that \(r\in V(Q_y)\).
Hence \(r\in V(F_x)\cap V(F_y)\), so \(xy\in E(G)\).
The same argument joins \(u\) to the first recorded vertex and the last recorded vertex to \(v\).

These vertices are distinct: the selected hyperedges are vertex disjoint in \(U\), and the added \(i_e\) are distinct elements of \(I\).
The deletion of \(Z\) prevents endpoints in \(U\) from being selected, while the choice of \(P\) prevents endpoints in \(I\) from appearing as \(i_e\).

For each missing \(x\in V(G)\), choose a vertex \(w\in V(Q_x)\).
At a visit of \(W\) to \(w\), insert together all missing vertices for which \(w\) was chosen.
At \(s\), use the position immediately after \(u\); at \(t\), use the position immediately before \(v\).
This preserves the path because the inserted vertices and their two neighbors have represented subtrees containing \(w\), and hence form a clique in \(G\).
\end{proof}

\subsection{Separator Construction}
We convert vertex covers of the \(A_e\) into controlled separators.

\begin{lemma}
\label{lem:cover-to-separator}
For \(C\subseteq U\), set
\[
  F_C:=\{e\in E(T):C\text{ is a vertex cover of }A_e\}.
\]
If \(F_C\ne\emptyset\), then there is a separating set \(S\subseteq V(G)\) such that
\[
  5c(G-S)-|S|\ge5+q(F_C)-|C|.
\]
\end{lemma}

\begin{proof}
Let
\[
\begin{aligned}
  B&:=\{e\in F_C:e\text{ is black}\},\\
  D&:=\{e\in F_C:e\text{ is gray and shares no vertex with an edge of }B\}.
\end{aligned}
\]
Let \(\lambda\) be the total number of endpoints of edges in \(B\) having degree at most two in \(T\).

For each gray edge \(e\in F_C\setminus D\), choose an endpoint that \(e\) shares with an edge of \(B\).
These endpoints are distinct and counted by \(\lambda\), since gray edges are vertex disjoint and have degree-two endpoints.
Set
\[
  S:=C\cup\{i_e:e\in D\}.
\]
Since \(C\subseteq U\) and \(\{i_e:e\in D\}\subseteq I\),
\begin{equation}
\label{eq:separator-counts}
  q(F_C)\le2|B|+\lambda+|D|,
  \qquad |S|=|C|+|D|.
\end{equation}

Let \(r\) be the number of components of \(T-B\) containing a vertex of degree at most two in \(T\).
If \(b_i\) denotes the number of edges of \(B\) having exactly \(i\) such endpoints, then Lemma~\ref{lem:kk-tree-estimate} gives
\begin{equation}
\label{eq:tree-component-bound}
  5(r-1)\ge2b_0+3b_1+5b_2
  \ge2|B|+\lambda.
\end{equation}

Deleting \(D\) from these \(r\) components of \(T-B\) produces \(r+|D|\) regions of \(T-(B\cup D)\).
We show that every region contains a vertex of \(Q_x\) for some \(x\in V(G)\setminus S\).
Fix a region \(R\), and let \(K\) be the component of \(T-B\) containing \(R\).
By the definition of \(r\), \(K\) contains a vertex \(t\) of degree at most two in \(T\).
By construction, \(t\in V(Q_i)\) for some \(i\in I\): a degree-two vertex remains after suppression only if it is an endpoint of some \(F_i\); for a leaf, the conclusion follows from the minimality of \(T_0\) and the maximality of \(I\).
If \(K\) contains no edge of \(D\), then \(t\in R\) and \(i\) does not belong to \(\{i_e:e\in D\}\); hence \(i\notin S\).
If \(K\) contains an edge of \(D\), then \(R\) contains an endpoint \(t\) of some \(e\in D\).
The other edge \(f\) incident with \(t\) is black since gray edges are vertex disjoint.
Since \(e\) shares no vertex with an edge of \(B\), we have \(f\notin F_C\).
Thus \(C\) does not cover \(A_f\), so \(A_f\) has a hyperedge disjoint from \(C\).
By the definition of \(A_f\), this hyperedge contains a resource \(x\) with \(t\in V(Q_x)\).
Since \(x\in U\setminus C\), we have \(x\notin S\).

For each of the \(r+|D|\) regions \(R\), fix a vertex \(x_R\in V(G)\setminus S\) such that \(Q_{x_R}\) meets \(R\).
Note that \(V(G)=U\sqcup I\).
Then, for every \(x\in V(G)\setminus S\), the subtree \(Q_x\) crosses no edge of \(B\cup D\): if \(x\in U\setminus C\), then \(Q_x\) cannot cross an edge \(e\in B\cup D\subseteq F_C\), since the loop \(\{x\}\) of \(A_e\) would have to be covered by \(C\), contradicting \(x\notin C\); if \(x\in I\setminus S\), then \(Q_x\) is a singleton or a gray edge, so it crosses no black edge, and crossing an edge \(e\in D\) would give \(x=i_e\in S\), again a contradiction.
Since \(Q_x\) is connected, it meets at most one region.

Now consider two distinct regions \(R_1,R_2\).
The path in \(T\) between them contains an edge \(e\in B\cup D\).
If \(e\in B\), Lemma~\ref{lem:kk-separation} places \(x_{R_1}\) and \(x_{R_2}\) in different components of \(G-C\), and hence of \(G-S\).
If \(e\in D\), Lemma~\ref{lem:kk-separation} again separates \(x_{R_1}\) and \(x_{R_2}\) in \(G-(C\cup\{i_e\})\), and hence in \(G-S\).
Thus distinct regions give distinct components of \(G-S\), so
\begin{equation}
\label{eq:region-component-count}
  c(G-S)\ge r+|D|.
\end{equation}
By \eqref{eq:region-component-count}, \eqref{eq:separator-counts}, and \eqref{eq:tree-component-bound},
\[
\begin{aligned}
  5c(G-S)-|S|&\ge 5(r+|D|)-(|C|+|D|)
  =5r+4|D|-|C|\\
  &\ge5+2|B|+\lambda+4|D|-|C|\\
  &\ge5+q(F_C)+3|D|-|C|
  \ge5+q(F_C)-|C|.
\end{aligned}
\]
If \(B\ne\emptyset\), \eqref{eq:tree-component-bound} implies \(r\ge2\), so \eqref{eq:region-component-count} gives \(c(G-S)\ge2\).
If \(B=\emptyset\), then \(T-B=T\), so \(r=1\), and \(D=F_C\ne\emptyset\); again \eqref{eq:region-component-count} gives \(c(G-S)\ge2\).
Thus \(S\) is separating.
\end{proof}

\subsection{Proof of Theorem~\ref{thm:five-tough-hamilton-connected} and Further Discussion}
The proof turns a violated Hall inequality into a small vertex cover and then into a separator contradicting \(5\)-toughness.
\begin{proof}[Proof of Theorem~\ref{thm:five-tough-hamilton-connected}]
Fix distinct vertices \(u,v\in V(G)\), and let \(Z:=\{u,v\}\setminus I\).
If \(T\) has no edge, then \(G\) is complete.
Assume otherwise and consider the tree-structured hypergraph system
\[
  \mathcal A_Z:=(A_e-Z)_{e\in E(T)}
\]
on the resource set \(U\setminus Z\) with multiplicity function \(q\).

Suppose that \(\mathcal A_Z\) admits no \(q\)-SDR.
By Theorem~\ref{thm:qSDR}, some nonempty \(F_0\subseteq E(T)\) satisfies
\[
  \nu(A(F_0)-Z)\le q(F_0)-1.
\]
Let \(L\) be the resources carrying loops in \(A(F_0)-Z\).
Every remaining support lies in one component of \(T-F_0\): if \(Q_x\) crossed an edge \(tu\in F_0\), then \(x\in P_t\cap P_u\), so \(\{x\}\) would be a loop and \(x\in L\).
Contract each component of \(T-F_0\) to a vertex.
Deleting \(L\) therefore leaves only pairs in \(H:=A(F_0)-(Z\cup L)\), and \(H\) is bipartite.
K\H{o}nig's theorem and the loops on \(L\) give
\[
  \nu(A(F_0)-Z)=|L|+\nu(H)
  =|L|+\tau(H)
  =\tau(A(F_0)-Z).
\]
Hence \(A(F_0)-Z\) has a vertex cover \(C_0\) with
\begin{equation}
\label{eq:small-cover}
  |C_0|\le q(F_0)-1.
\end{equation}
Set \(C:=C_0\cup Z\).
Then \(C\) covers \(A(F_0)\), so \(F_0\subseteq F_C\), and \eqref{eq:small-cover} gives
\[
  q(F_C)-|C|\ge q(F_0)-|C_0|-|Z|
  \ge1-|Z|.
\]
Lemma~\ref{lem:cover-to-separator} gives a separating set \(S\) with
\[
  5c(G-S)-|S|\ge6-|Z|
  \ge4.
\]
Thus \(|S|<5c(G-S)\), contradicting the \(5\)-toughness of \(G\).

Therefore \(\mathcal A_Z\) admits a \(q\)-SDR.
Lemma~\ref{lem:endpoint-construction} gives a Hamilton path from \(u\) to \(v\).
Since \(u,v\) were arbitrary, \(G\) is Hamilton-connected.
\end{proof}

Chen, Jacobson, K\'{e}zdy and Lehel proved the Hamiltonicity bound \(18\) in 1998~\cite{ChenJacobsonKezdyLehel}, and Kabela and Kaiser lowered it to \(10\) while proving Hamilton-connectedness in 2017~\cite{KabelaKaiser}.
Theorem~\ref{thm:five-tough-hamilton-connected} lowers the sufficient toughness bound to \(5\) while retaining Hamilton-connectedness.
On the other hand, Bauer, Broersma and Veldman constructed nonhamiltonian chordal graphs with toughness arbitrarily close to \(7/4\)~\cite{BauerBroersmaVeldman}.
This motivates the following conjecture.
\footnote{The \(2\)-tough Hamiltonicity bound for chordal graphs was posed as an open question in 2006~\cite{BauerBroersmaSchmeichel}.}

\begin{conjecture}
\label{conj:two-tough}
Every \(2\)-tough chordal graph on at least three vertices is Hamilton-connected.
\end{conjecture}

\section{Other Consequences of the Hall Theorem}
\subsection{The min--max formula}

\begin{corollary}
\label{cor:deficiency}
Let \(\mathcal A=(A_e)_{e\in E(T)}\) be a tree-structured hypergraph system, and let \(q:E(T)\to\mathbb Z_{>0}\) be a multiplicity function.
A \emph{partial \(q\)-SDR} for \(\mathcal A\) chooses, for each \(e\in E(T)\), a set \(M_e\subseteq E(A_e)\) with \(|M_e|\le q_e\), such that all chosen hyperedges are pairwise vertex disjoint.
Let \(\rho(\mathcal A,q)\) be the maximum of \(\sum_{e\in E(T)}|M_e|\) over all partial \(q\)-SDRs.
Then
\[
\begin{aligned}
  \rho(\mathcal A,q)&=\min_{F\subseteq E(T)}\left(q(E(T)\setminus F)+\nu(A(F))\right)\\
  &=q(E(T))-\max_{F\subseteq E(T)}\left(q(F)-\nu(A(F))\right).
\end{aligned}
\]
\end{corollary}

\begin{proof}
Write \(E:=E(T)\), \(Q:=q(E)\), and
\[
  \Delta:=\max_{F\subseteq E}\left(q(F)-\nu(A(F))\right).
\]
For every \(F\subseteq E\), a partial \(q\)-SDR contains at most \(q(E\setminus F)\) representatives assigned to indices in \(E\setminus F\), while those assigned to indices in \(F\) form a matching in \(A(F)\).
Therefore
\[
  \rho(\mathcal A,q)\le q(E\setminus F)+\nu(A(F))
  =Q-\left(q(F)-\nu(A(F))\right).
\]
Taking the minimum over \(F\) gives \(\rho(\mathcal A,q)\le Q-\Delta\).

Add a set \(W\) of \(\Delta\) new resources, and give every \(w\in W\) the support \(Q_w:=T\).
For every \(e\in E\), obtain \(A_e^+\) from \(A_e\) by adding one loop at every resource of \(W\) and no other hyperedge containing a resource of \(W\).
Write
\[
  \mathcal A^+:=(A_e^+)_{e\in E},
  \qquad E(A^+(F)):=\bigcup_{e\in F}E(A_e^+).
\]
The family \(\mathcal A^+\) is a tree-structured hypergraph system, and for every nonempty \(F\subseteq E\),
\[
  \nu(A^+(F))=\nu(A(F))+\Delta
  \ge q(F).
\]
Theorem~\ref{thm:qSDR} gives a \(q\)-SDR of size \(Q\) for \(\mathcal A^+\).
At most \(\Delta\) of its representatives consume resources in \(W\).
Deleting those representatives leaves a partial \(q\)-SDR for \(\mathcal A\) of size at least \(Q-\Delta\).
Thus \(\rho(\mathcal A,q)\ge Q-\Delta\).
Together with the upper bound, this proves the corollary.
\end{proof}

\subsection{Existence of a Polynomial Time Algorithm}
\begin{corollary}
\label{cor:polynomial-algorithm}
Given a tree-structured hypergraph system \(\mathcal A\) and a multiplicity function \(q\), there is a polynomial time algorithm to decide whether \(\mathcal A\) admits a \(q\)-SDR.
\end{corollary}
\begin{proof}
We record how many hyperedges of a matching are used as representatives for each \(A_e\).
For a matching \(M\) in \(A(E(T))\), a labelling is a map \(\ell:M\to E(T)\) such that \(a\in E(A_{\ell(a)})\) for every \(a\in M\).
Its count vector is \(p=(p_e)_{e\in E(T)}\), where \(p_e:=|\ell^{-1}(e)|\).
For \(F\subseteq E(T)\), write \(p(F):=\sum_{e\in F}p_e\).
Let \(\mathcal P\) consist of the nonnegative vectors that are bounded coordinatewise by a convex combination of these count vectors.

We claim that \(\mathcal A\) admits a \(q\)-SDR if and only if \(q\in\mathcal P\).
One direction is immediate.

Conversely, suppose that \(q\in\mathcal P\).
By definition of \(\mathcal P\), there are count vectors \(p^{(1)},\ldots,p^{(k)}\) and nonnegative coefficients \(\lambda_1,\ldots,\lambda_k\), with \(\sum_{i=1}^k\lambda_i=1\), such that \(q\le\sum_{i=1}^k\lambda_i p^{(i)}\) coordinatewise.
For every \(i\) and \(F\subseteq E(T)\), the \(p^{(i)}(F)\) representatives whose labels lie in \(F\) form a matching in \(A(F)\), so \(p^{(i)}(F)\le\nu(A(F))\).
Therefore, for every \(F\subseteq E(T)\),
\[
  q(F)\le\sum_{i=1}^k\lambda_i p^{(i)}(F)\le\nu(A(F)).
\]
Theorem~\ref{thm:qSDR} therefore gives a \(q\)-SDR.

It remains to decide whether \(q\in\mathcal P\) in polynomial time.
Every count vector is integral and lies in \([0,|U|]^{E(T)}\), so \(\mathcal P\) is a well-described bounded rational polytope.
By the optimization--separation equivalence~\cite[Theorem~6.4.9]{GLS}, it suffices to optimize linear functions over \(\mathcal P\) in polynomial time.
For \(w\in\mathbb R^{E(T)}\), first replace every negative \(w_e\) by \(0\), which does not change the optimum. 
Assign each representative \(a\) weight \(\max\{w_e:a\in E(A_e)\}\), and replace every loop by an edge to a new private vertex.
After finding a maximum-weight matching~\cite{EdmondsMatching}, discard its zero-weight edges and label each remaining representative by an index attaining its weight.
This finds a count vector maximizing \(w\cdot x\) over \(\mathcal P\) in polynomial time, completing the proof.
\end{proof}

\begin{remark}
A standard projected subgradient argument~\cite[Section~3.1]{BubeckConvex}, using the maximum-weight matching oracle above, also gives a deterministic strongly polynomial feasibility algorithm.
\end{remark}

\subsection{A polynomial time solvable case of full rainbow matching}

The problem of finding a matching containing exactly one edge of each color is known as the \emph{full rainbow matching} problem~\cite{GaoRamaduraiWanlessWormald}.
This problem is NP-complete even in bipartite multigraphs.
\footnote{Given a \(3\)-dimensional matching \(\mathcal M\subseteq X\times Y\times Z\), where \(X,Y,Z\) have equal size, form a bipartite multigraph with classes \(X,Y\) by replacing each triple \((x,y,z)\in\mathcal M\) with an edge \(xy\) of color \(z\). A matching containing one edge of each color corresponds exactly to a set of \(|Z|\) pairwise disjoint triples~\cite[Problem~SP1, p.~221]{GareyJohnson}.}
The following corollary identifies a polynomial time solvable case of this problem.

\begin{corollary}
\label{cor:prescribed-colors}
Let \(T\) be a finite tree, let \((U_t)_{t\in V(T)}\) be pairwise disjoint finite sets, and, for each edge \(e=tu\), let \(H_e\) be a bipartite graph with classes \(U_t\) and \(U_u\).
Regard \(E(H_e)\) as the color class \(e\).
Let \(q:E(T)\to\mathbb Z_{>0}\).
There is a matching \(M\subseteq\bigcup_{e\in E(T)}E(H_e)\) such that
\[
  |M\cap E(H_e)|=q_e \qquad \text{for every color } e\in E(T)
\]
if and only if
\[
  \nu\left(\bigcup_{e\in F}E(H_e)\right)\ge\sum_{e\in F}q_e \qquad \text{for every \(F\subseteq E(T)\)}.
\]
Moreover, the existence of such a matching can be tested in polynomial time.
\end{corollary}

\begin{proof}
Take \(U:=\bigcup_{t\in V(T)}U_t\) as the resource set and give each \(x\in U_t\) the singleton support \(Q_x:=\{t\}\).
For every \(e\in E(T)\), let \(A_e\) be the hypergraph on \(U\) with edge set \(E(H_e)\).
Since the sets \(U_t\) are pairwise disjoint, the corresponding bag \(P_t\) is \(U_t\).
This gives a tree-structured hypergraph system with no loops, and its \(q\)-SDRs are exactly the matchings satisfying the prescribed color requirements.
The characterization follows from Theorem~\ref{thm:qSDR}, and the algorithm follows from Corollary~\ref{cor:polynomial-algorithm}.
\end{proof}

The corollary gives a multiplicity version of the full rainbow matching problem.
Taking \(q_e=1\) for every color \(e\) gives the full rainbow matching case.

\section{Acknowledgments}
Claude Opus was used to search for references. 
AI tools were used to refine the writing after the draft was completed.

\printbibliography
\end{document}